\documentclass[a4paper,10pt]{article}
\usepackage{authblk}

\usepackage[ngerman,english]{babel}
\usepackage[latin1]{inputenc}
\usepackage{csquotes}
\usepackage{amsfonts,amsmath,amsthm}
\usepackage{empheq}
\usepackage[titletoc,title]{appendix}
\usepackage[backend=bibtex8,doi=false,eprint=false,firstinits=true,isbn=false,style=numeric-comp,maxnames=99]{biblatex}
\makeatletter
\def\blx@maxline{77}
\makeatother
\bibliography{bibl.bib}
\AtEveryBibitem{\clearlist{language}}

\usepackage{cases}
\usepackage{mathabx}
\usepackage{bbm}
\usepackage{xfrac}
\usepackage{fancyhdr}
\usepackage{color}
\usepackage[colorlinks]{hyperref}
\definecolor{blue75}{rgb}{0,0,.75}
\definecolor{green75}{rgb}{0,.75,0}
\hypersetup{colorlinks=true, urlcolor=blue75,linkcolor=blue75,citecolor=green75,pdfstartview=FitB,bookmarksopen=true,bookmarksopenlevel=1}
\usepackage[a4paper, left=2.5cm, right=2.5cm, top=2.5cm,bottom=2cm]{geometry}
\usepackage{constants}
\newcommand{\parenthezises}[1]{\arabic{#1}}
\newconstantfamily{C}{
symbol=C,
format=\parenthezises,
reset={section}
}
\newconstantfamily{M}{
symbol=M,
format=\parenthezises,
reset={section}
}
\newconstantfamily{B}{
symbol=B,
format=\parenthezises,
reset={section}
}
\usepackage{enumerate}
\usepackage{graphicx}
\graphicspath{ {images/} }
\usepackage{wrapfig}
\usepackage{figbib}
\allowdisplaybreaks
\begin{document}

\newcommand{\cmg}{\color{magenta}}
\newcommand{\D}{\mathbb{D}}
\newcommand{\R}{\mathbb{R}}
\newcommand{\N}{\mathbb{N}}

\def\diam{\operatorname{diam}}
\def\dist{\operatorname{dist}}
\def\diver{\operatorname{div}}
\def\ess{\operatorname{ess}}
\def\inner{\operatorname{int}}
\def\osc{\operatorname{osc}}
\def\sign{\operatorname{sign}}
\def\supp{\operatorname{supp}}
\newcommand{\BMO}{BMO(\Omega)}
\newcommand{\LOne}{L^{1}(\Omega)}
\newcommand{\LOnen}{(L^{1}(\Omega))^d}
\newcommand{\LTwo}{L^{2}(\Omega)}
\newcommand{\Lq}{L^{q}(\Omega)}
\newcommand{\Lp}{L^{2}(\Omega)}
\newcommand{\Lpn}{(L^{2}(\Omega))^d}
\newcommand{\LInf}{L^{\infty}(\Omega)}
\newcommand{\HOneO}{H^{1,0}(\Omega)}
\newcommand{\HTwoO}{H^{2,0}(\Omega)}
\newcommand{\HOne}{H^{1}(\Omega)}
\newcommand{\HTwo}{H^{2}(\Omega)}
\newcommand{\HmOne}{H^{-1}(\Omega)}
\newcommand{\HmTwo}{H^{-2}(\Omega)}

\newcommand{\LlogL}{L\log L(\Omega)}

\def\avint{\mathop{\,\rlap{-}\!\!\int}\nolimits}

\newtheorem{Theorem}{Theorem}[section]
\newtheorem{Assumption}[Theorem]{Assumptions}
\newtheorem{Corollary}[Theorem]{Corollary}
\newtheorem{Convention}[Theorem]{Convention}
\newtheorem{Definition}[Theorem]{Definition}
\newtheorem{Example}[Theorem]{Example}
\newtheorem{Lemma}[Theorem]{Lemma}

\theoremstyle{definition}
\newtheorem{Remark}[Theorem]{Remark}
\newtheorem{Notation}[Theorem]{Notation}

\newtheorem{proofpart}{Step}
\makeatletter
\@addtoreset{proofpart}{Theorem}
\makeatother
\numberwithin{equation}{section}
\renewcommand\Affilfont{\itshape\small}

\title{On a model for epidemic spread with interpopulation contact and repellent taxis}
\author{{C}higanga {S}amson {R}uoja\footnote{supported by  the German Academic Exchange Service (DAAD), grant \#57191713}}
\affil{Dar es Salaam University College of Education\protect\\
 P.O.Box 2329, Dar es Salaam,  Tanzania\protect\\
  e-mail: {samson.chiganga@yahoo.com}}

\author{{C}hristina {S}urulescu}
\affil{Technische Universit\"at Kaiserslautern, Felix-Klein-Zentrum f\"ur Mathematik\protect\\
Paul-Ehrlich-Str. 31, 67663 Kaiserslautern, Germany\protect\\
  e-mail: {A.Zhigun@qub.ac.uk}}
\date{}

\author{Anna~Zhigun\footnote{corresponding author}}
\affil{Queen's University Belfast, School of Mathematics and Physics\protect\\
University Road, Belfast BT7 1NN, Northern Ireland, UK\protect\\
  e-mail: {surulescu@mathematik.uni-kl.de}}
  
\maketitle
\begin{abstract}
 We study a PDE model for dynamics of susceptible-infected interactions.   The dispersal of susceptibles is via diffusion and repellent taxis as they move 
away from the increasing density of infected. 
The  diffusion of infected is a nonlinear, possibly degenerating term in nondivergence form. 
%
We prove the existence of so-called weak-strong solutions in 1D  for a positive susceptible initial population. For dimension $N\ge 2$ and nonnegative susceptible initial density we show 
the existence of supersolutions. Numerical simulations are performed for different scenarios and illustrate the space-time behaviour of solutions.

\end{abstract}

\section{Introduction}\label{sec:intro}

\noindent
There exists by now a vast literature about models of epidemic spread, most of which take the form of ODE systems. The heterogeneity of space, however, can play an essential role in the dynamics 
of an infectious disease, as the environmental conditions can differ from one site to the other and moreover the individuals in the populations can move in space. Therefore various settings accounting for 
both space and time variability have been proposed, among the first being the contact model with diffusion in \cite{kendall}, followed by \cite{Hadeler1975,Hadeler1997,Hosono1995,Kaellen-AM} and many others, mainly performing traveling wave analysis. \footnote{In this context we are interested only in reaction-diffusion models; population balance models accounting for further structures and featuring integro-differential PDEs (see e.g. \cite{magal} and the references therein) are not addressed here.} During the last two decades the investigation of well-posedness and qualitative properties of solutions to reaction-diffusion PDE systems describing epidemics has attracted increasing interest, see e.g. \cite{allen,deng,huang,wang-lw,wu}; we also refer to \cite{fitzgibb,webb} for earlier works. Some of the more recent models \cite{cui16,cui17,jiang,kuto} were extended to account for at least one of the interacting populations having a motility bias in a certain direction (e.g., due to environmental influences like fluid or air flow), which leads to a drift term supplementing the diffusion and the source/decay terms. 

\noindent
A model with linear cross-diffusion of susceptibles has been considered in \cite{Sun}, while \cite{berres} proposed a numerical scheme to handle a nonlinear one.  Such terms are included into the  epidemic model in order to account  for the response of an active population of susceptibles toward the other population's degree of infectiveness. The former would typically try to avoid the latter by biasing its movement in the direction opposite to the gradient of infected population density. How effective this avoidance is depends, of course, on the amount of susceptibles and infected in the system and their respective ratio: A small amount of infected in an overwhelming susceptible population will pass -at least for a while- unnoticed (\textit{carelessness}). In the other extreme, a large density of infected in a comparatively rather small population of susceptibles will lead to \textit{fatalism} (very low possibility of avoidance). At moderate densities and ratios the avoidance mechanism can be quite effective, leading to patterns and (local) phase separation, see the simulations in \cite{berres,Sun}. The corresponding advection term occurring in the PDE for the density of susceptibles can be interpreted as a repellent taxis term, in analogy with models of (chemo)repellence involving (chemo)taxis terms with a sign opposite to that usual for Keller-Segel type models.

\noindent
In this work we propose and investigate a reaction-diffusion model for infection spread with contact and with repellent taxis, which combines cross-diffusion (in the sense mentioned above) with the influence of contacts between susceptibles and infected on the self-diffusion of the latter. To our knowledge, previous epidemic models have either one or the other of these features, see \cite{kendall,postnikov,Sun}; the model in \cite{berres} involves self-diffusion of infected with self-contact. The mentioned studies in \cite{postnikov,Sun,berres} mostly focus upon computational aspects. They do not include any proofs of existence of solutions to such settings. Here we aim at addressing this analytical issue for the announced model extension.

\noindent
The paper is structured as follows: {\it Section \ref{sec:model}} specifies the model to be investigated, along with the mathematical challenges arising from its structure, and introduces a family of regularized problems approximating the actual PDE system of interest. {\it Section \ref{WPRepDege}} is dedicated to the analysis of this regularized problem, followed in {\it Section \ref{Ex1D}} by the existence proof for solutions to the full model in 1D and with positive initial density of susceptibles. {\it Section \ref{Gensuper}} contains the existence proof for supersolutions of our model in space dimension $N\ge 2$ and for nonnegative initial density of susceptibles. Finally, {\it Section \ref{sec:nums}} provides some numerical simulations to illustrate the behavior of the model for a few different scenarios of the epidemics. Some concluding remarks are provided as well.

\section{Model setup}\label{sec:model}

\noindent
We consider the following model for the evolution of susceptibles ($S$) and infected ($I$): 
\begin{subequations}\label{RepDeg}
 \begin{empheq}[left={ \empheqlbrace\ }]{align}
  &\partial_t S=\nabla\cdot\left(\nabla S+\chi(S)S\nabla I\right)+f(S,I)&&\text{ in }\R^+\times\Omega,\label{EqS}\\
  &\partial_t I=S\Delta I+g(S,I)&&\text{ in }\R^+\times\Omega,\label{EqI}\\
  &\partial_{\nu} S=\partial_{\nu} I=0&&\text{ in }\R^+\times\partial\Omega,\label{bc}\\
  &S(0,\cdot)=S_0,\qquad I(0,\cdot)=I_0&&\text{ in }\Omega.
 \end{empheq}
\end{subequations}
Here $\Omega$ is a bounded domain in $\R^N$, $N\in\N$, with the corresponding outer normal unit vector $\nu$ on $\partial\Omega$. 

\noindent
The equations describe the interactions of the two populations, whereby $S$ performs linear diffusion and repellent taxis in the sense mentioned in {\it Section \ref{sec:intro}}: Susceptibles tend to avoid the  infected. The efficiency of avoidance is characterized by the function $\chi $ which in analogy to chemotaxis models will be called in the following \textit{tactic sensitivity}. It may in fact depend both on $S$ and $I$, however for our analysis we choose it in the form 
\begin{align}
 &\chi(S)=K(1-S),\nonumber
\end{align}
which accounts for the crowding effect: amidst a large mass of susceptibles their awareness of infectives is reduced ('drowned'). 
In particular, the threshold value for $S$ corresponding to a tight packing state is assumed to be normalised to $S_{max}=1$.

\noindent
The first term on the right hand side of \eqref{EqI} describes self-diffusion with interpopulation contact, similarly to the model in \cite{kendall}.  
Further, we modify the interaction terms having the usual form $SI$ to account for an infectiveness threshold with a limitation given by the total population. Thus, following e.g., \cite{Sun} or \cite{allen} we choose 
\begin{equation*}
f(S,I)=-\lambda_S\frac{SI}{S+I}+\mu_SS(1-S),\qquad g(S,I)=\lambda_I\frac{SI}{S+I}-\mu_II,
\end{equation*}
where the second terms in these expressions describe as usual logistic growth of susceptibles and linear removal of infected, respectively. \footnote{The function $f_1(S,I)=\frac{SI}{S+I}$ is Lipschitz with respect to $S$ and $I$ in the open first quadrant, hence its definition can be extended to the closure of that set by letting it be zero when either $S=0$ or $I=0$.} 
\noindent
The removed (including dead and recovered) population can be described by 
\begin{equation*}
	\partial_tR=\mu _II.
\end{equation*}
This equation is decoupled from \eqref{RepDeg}, thus not contributing to the dynamics of $(S,I)$ and will therefore be ignored in the following.

\paragraph{Analytical challenges.} System \eqref{RepDeg} combines several effects which jointly make the analysis challenging:
\begin{enumerate}[(i)]
 \item \eqref{EqS} is in divergence form, while \eqref{EqI} is not, and they are strongly coupled;
 \item equation \eqref{EqS} for $S$ includes a potentially destabilising chemotaxis transport term, in this case in the direction opposite to $\nabla I$;
 \item equation \eqref{EqI} for $I$ features a non-standard degeneracy occurring on the zero level set of the variable $S$.  
\end{enumerate}
System \eqref{RepDeg} can be seen as a {\it formal} limit as $\varepsilon\rightarrow 0$ of the following family of regularised problems:
\begin{subequations}\label{RepDege}
 \begin{empheq}[left={ \empheqlbrace\ }]{align}
  &\partial_t S_{\varepsilon}=\nabla\cdot\left(\nabla S_{\varepsilon}+\chi(S_{\varepsilon})S_{\varepsilon}\nabla I_{\varepsilon}\right)+f(S_{\varepsilon},I_{\varepsilon})&&\text{ in }\R^+\times\Omega,\label{EqSe}\\
  &\partial_t I_{\varepsilon}=(\varepsilon+S_{\varepsilon})\Delta I_{\varepsilon}+g(S_{\varepsilon},I_{\varepsilon})&&\text{ in }\R^+\times\Omega,\label{EqIe}\\
  &\partial_{\nu} S_{\varepsilon}=\partial_{\nu} I_{\varepsilon}=0&&\text{ in }\R^+\times\partial\Omega,\label{bce}\\
  &S_{\varepsilon}(0,\cdot)=S_0,\qquad I_{\varepsilon}(0,\cdot)=I_0&&\text{ in }\Omega.
 \end{empheq}
\end{subequations}
Thereby a small number $\varepsilon\in(0,1]$ added to the diffusion coefficient of $I$ eliminates the degeneracy issue. Standard tools can be used (see {\it Section \ref{WPRepDege}} below) in order to prove the existence of solutions for \eqref{RepDege}.  
This observation naturally leads to an attempt to apply the standard compactness method which is based on  establishing uniform w.r.t. $\varepsilon$ a priori estimates for $S_{\varepsilon}$, $I_{\varepsilon}$, and their derivatives in suitable Bochner spaces and utilising some known compact embeddings and other necessary results in order to prove the existence of a sequence which converges to  some weak solution of the non-perturbed system. In general, however, owing to the sort of degeneracy present in the original system, it seems difficult, if not impossible,  to get a priori bounds which would allow to pass rigorously to the limit in the term $S_{\varepsilon}\Delta I_{\varepsilon}$ in order to regain $S\Delta I$. 

\noindent
In order to obtain our existence results, we assume that
\begin{subequations}\label{AssIni}
\begin{alignat}{3}
 &S_0\in L^{\infty}(\Omega),&&\qquad 0\leq S_0\leq 1,\\
 &I_0\in W^{1,\infty}(\Omega),&&\qquad I_0\geq0. 
\end{alignat}
\end{subequations}
\noindent
In this work we consider first the special situation when 
\begin{align}
N=1 \qquad\text{and}\qquad \inf_{\Omega}S_0\in(0,1]\nonumber
\end{align}
 and prove in {\it Section \ref{Ex1D}} that under these assumptions a solution does exist. 
 In {\it Section \ref{Gensuper}} we then turn to the general case of an arbitrary space dimension and without a positive lower bound for $S_0$. In this case we are able to establish the existence of a weak  supersolution (see {\it Definition \ref{DefGenSuperSol}} below). 

\begin{Remark}[Notation]
Throughout the paper we make the following useful conventions:
\begin{enumerate}
 \item For any index $i$, a quantity $C_i$ denotes a positive constant or function;
 \item Dependence upon such parameters as: the space dimension $N$, domain 
$\Omega$, constants $K,\lambda_I,\lambda_S,\mu_I,\mu_S$, and the norms 
of the initial data $S_0$ and $I_0$ is mostly {\bf not}
indicated in an explicit way;
\item We assume the reader to be familiar with the conventional Lebesgue, Sobolev, and Bochner spaces and standard results concerning them. We denote $\left<\cdot,\cdot\right>$ the duality paring between $H^1(\Omega)$ and its dual $(H^1(\Omega))'$.
\end{enumerate}
\end{Remark}
\section{Analysis of the  regularized system \texorpdfstring{\eqref{RepDege}}{}}\label{WPRepDege}
\subsection{A priori estimates and compactness}
To begin with, we establish several necessary a priori estimates for \eqref{RepDege}. 
\begin{Lemma}\label{Lembnd} Let a pair of measurable functions $(S_{\varepsilon},I_{\varepsilon}):\R_0^+\times \overline{\Omega}\rightarrow[0,1]\times \R_0^+$ be a sufficiently regular solution to system \eqref{RepDege}. Then it satisfies the following estimates:
\begin{align}
 &\|\nabla I_{\varepsilon}\|_{L^{\infty}(\R^+;L^2(\Omega))}\leq \C,\label{estnabIe}\\
 &\|I_{\varepsilon}\|_{L^2(0,T;L^2(\Omega))}\leq \C(T),\label{estIe}\\
 &\sqrt{\varepsilon}\left\|\Delta I_{\varepsilon}\right\|_{L^2(0,T;L^2(\Omega))}+\left\|\sqrt{S_{\varepsilon}}\Delta I_{\varepsilon}\right\|_{L^2(0,T;L^2(\Omega))}\leq \C(T)\label{estDeltaIe}\\
 &\|\partial_t I_{\varepsilon}\|_{L^2(0,T;L^2(\Omega))}\leq \C(T),\label{estptIe}\\
 &\|\nabla S_{\varepsilon}\|_{L^2(0,T;L^2(\Omega))}\leq \C(T),\label{estnabSe}\\
 &\|\partial_t S_{\varepsilon}\|_{L^2(0,T;(H^1(\Omega))')}\leq \C(T)\label{estSet}
\end{align}
\end{Lemma}
\begin{proof}
 Testing \eqref{EqIe} with $-\Delta I_{\varepsilon}$ and { using Young's inequality} 
 we obtain that
 \begin{align}
  \frac{1}{2}\frac{d}{dt}\|\nabla I_{\varepsilon}\|_{L^2(\Omega)}^2=&-\varepsilon\left\|\Delta I_{\varepsilon}\right\|_{L^2(\Omega)}^2-\left\|\sqrt{S_{\varepsilon}}\Delta I_{\varepsilon}\right\|_{L^2(\Omega)}^2-\int_{\Omega}g(S_{\varepsilon},I_{\varepsilon})\Delta I_{\varepsilon}\,dx\nonumber\\
  \leq&-\varepsilon\left\|\Delta I_{\varepsilon}\right\|_{L^2(\Omega)}^2-\left\|\sqrt{S_{\varepsilon}}\Delta I_{\varepsilon}\right\|_{L^2(\Omega)}^2-\mu_I\|\nabla I_{\varepsilon}\|_{L^2(\Omega)}^2+\C\left\|\sqrt{S_{\varepsilon}}\Delta I_{\varepsilon}\right\|_{L^2(\Omega)}\nonumber\\
  \leq&-\varepsilon\left\|\Delta I_{\varepsilon}\right\|_{L^2(\Omega)}^2-\frac{1}{2}\left\|\sqrt{S_{\varepsilon}}\Delta I_{\varepsilon}\right\|_{L^2(\Omega)}^2-\mu_I\|\nabla I_{\varepsilon}\|_{L^2(\Omega)}^2+\C.\label{est1}
 \end{align}
Using the Gronwall lemma { and integrating with respect to time when necessary} we conclude from \eqref{est1} that estimates \eqref{estnabIe} and \eqref{estDeltaIe} hold. 
Testing \eqref{EqIe} with $I_{\varepsilon}$ and integrating over $\Omega$ we obtain with the Cauchy-Schwarz and Young inequalities that
\begin{align}
 &\frac{1}{2}\frac{d}{dt}\|I_{\varepsilon}\|_{L^2(\Omega)}^2\nonumber\\=&\int_{\Omega}\left(I_{\varepsilon}(\varepsilon+S_{\varepsilon})\Delta I_{\varepsilon}+I_{\varepsilon}g(S_{\varepsilon},I_{\varepsilon})\right) \,dx\nonumber\\
 \leq& \C\left(\sqrt{\varepsilon}\left\|\Delta I_{\varepsilon}\right\|_{L^2(\Omega)}+\left\|\sqrt{S_{\varepsilon}}\Delta I_{\varepsilon}\right\|_{L^2(\Omega)}\right)\|I_{\varepsilon}\|_{L^2(\Omega)}+\C\|I_{\varepsilon}\|_{L^2(\Omega)}-\mu_I\|I_{\varepsilon}\|_{L^2(\Omega)}^2\nonumber\\
 \leq& \C\left(\varepsilon\left\|\Delta I_{\varepsilon}\right\|_{L^2(\Omega)}^2+\left\|\sqrt{S_{\varepsilon}}\Delta I_{\varepsilon}\right\|_{L^2(\Omega)}^2+1\right)-\frac{\mu_I}{2}\|I_{\varepsilon}\|_{L^2(\Omega)}^2.\label{estI2}
\end{align}
Combining the Gronwall lemma with  \eqref{estDeltaIe} and \eqref{estI2} we obtain \eqref{estIe}. Altogether, estimates \eqref{estnabIe}-\eqref{estDeltaIe} allow to estimate the right-hand side of { \eqref{EqIe}}  yielding \eqref{estptIe}.

\noindent
Next, we test \eqref{EqSe} with $S_{\varepsilon}$ and integrate over $\Omega$, by parts where necessary. We thus obtain { by} using the Cauchy-Schwarz and Young inequalities and { the} estimate \eqref{estnabIe} that
\begin{align}
 \frac{1}{2}\frac{d}{dt}\|S_{\varepsilon}\|_{L^2(\Omega)}^2=&-\|\nabla S_{\varepsilon}\|_{L^2(\Omega)}^2+\int_{\Omega}-\chi(S_{\varepsilon})S_{\varepsilon}\nabla S_{\varepsilon}\cdot \nabla I_{\varepsilon}+S_{\varepsilon}f(S_{\varepsilon},I_{\varepsilon})\,dx\nonumber\\
 \leq&-\frac{1}{2}\|\nabla S_{\varepsilon}\|_{L^2(\Omega)}^2+\C.\label{est2}
\end{align}
Integrating \eqref{est2} over $(0,T)$ we obtain  \eqref{estnabSe}. Finally, \eqref{estnabIe} and \eqref{estnabSe} allow to estimate the right-hand side of \eqref{EqSe} yielding \eqref{estSet}. 
\end{proof}
\begin{Lemma}\label{LemConv}
 Let for each $\varepsilon\in(0,1]$  { the} pair   $(S_{\varepsilon},I_{\varepsilon}):\R_0^+\times \overline{\Omega}\rightarrow[0,1]\times \R_0^+$ be a sufficiently regular solution to system \eqref{RepDege}. Then there exists a sequence $\varepsilon_n\rightarrow 0$ and  a pair $(S,I)$ such that for any $T>0$
\begin{alignat}{3}
 &\nabla I_{\varepsilon_n}\underset{n\rightarrow\infty}{\overset{*}{\rightharpoonup}}\nabla I&&\qquad \text{in }L^{\infty}(0,T;L^2(\Omega)),\label{convnabIe}\\
 &\varepsilon_n\Delta I_{\varepsilon_n}\underset{n\rightarrow\infty}{\rightarrow}0&&\qquad \text{in }L^2(0,T;L^2(\Omega)),\\
 &\partial_t I_{\varepsilon_n}\underset{n\rightarrow\infty}{\rightharpoonup}\partial_t I&&\qquad \text{in } L^2(0,T;L^2(\Omega)),\\
 &I_{\varepsilon_n}\underset{n\rightarrow\infty}{\rightarrow}I&&\qquad \text{in }L^2(0,T;L^2(\Omega))\text{ and a.e.},\\
 &\nabla S_{\varepsilon_n}\underset{n\rightarrow\infty}{\rightharpoonup}\nabla S&&\qquad \text{in }L^2(0,T;L^2(\Omega)),\\
 &\partial_t S_{\varepsilon_n}\underset{n\rightarrow\infty}{\rightharpoonup} \partial_t S &&\qquad \text{in }L^2(0,T;(H^1(\Omega))'),\\
 &S_{\varepsilon_n}\underset{n\rightarrow\infty}{\rightarrow}S&&\qquad \text{in }L^2(0,T;L^2(\Omega))\text{ and a.e.}\label{convSet}
\end{alignat}
\end{Lemma}
\begin{proof}
 This  is  a direct consequence of  uniform estimates \eqref{estnabIe}-\eqref{estSet} combined with the Lions-Aubin lemma and { the} Banach-Alaoglu theorem.
\end{proof}
\subsection{Existence of solutions}
Since system \eqref{RepDege} couples two parabolic equations, one of which is in divergence form, while { the other} is not, the standard theory, e.g. from \cite{LSU} or \cite{Amann1}, seems not to be directly applicable. Still, existence of solutions can be obtained { by} using the well-established procedure based on the Schauder fixed point theorem. For the convenience of the reader, we state the corresponding existence result and sketch its proof.
We choose the following { notion} of a solution: 
\begin{Definition}[Weak-strong solution]\label{DefSole} We call a pair of measurable functions $(S_{\varepsilon},I_{\varepsilon}):\R_0^+\times \overline{\Omega}\rightarrow[0,1]\times \R_0^+$ a \underline{weak-strong solution} to system \eqref{RepDege} if: 
\begin{enumerate} 
  \item $\nabla S_{\varepsilon}\in L^2_{loc}(\R_0^+;L^2(\Omega))$, $\partial_t S_{\varepsilon}\in L^2_{loc}(\R_0^+;(H^1(\Omega))')$;
  \item $I_{\varepsilon}\in L^2_{loc}(\R_0^+;H^2(\Omega))$, $\nabla I_{\varepsilon}\in L^{\infty}(\R_0^+;L^2(\Omega))$, $\partial_t I_{\varepsilon}\in L^2_{loc}(\R_0^+;L^2(\Omega))$;
 \item  the pair $(S_{\varepsilon},I_{\varepsilon})$ is a weak solution to  \eqref{EqS} and a  strong solution to \eqref{EqI}, i.e., 
 for all $\varphi\in H^1(\Omega)$ it holds that
 \begin{subequations}\label{DefWStre}
 \begin{align}
  &\left<\partial_t S_{\varepsilon},\varphi\right>=-\int_{\Omega}(\nabla S_{\varepsilon}+\chi(S_{\varepsilon})S_{\varepsilon}\nabla I_{\varepsilon})\cdot\nabla\varphi\,dx+{ \int_{\Omega}} f(S_{\varepsilon},I_{\varepsilon})\varphi\,dx\qquad\text{a.e. in }\R_0^+,\label{EqSweake}\\
 &\partial_t I_{\varepsilon} =S_{\varepsilon}\Delta I_{\varepsilon} +g(S_{\varepsilon},I_{\varepsilon})\qquad\text{a.e. in }\R_0^+\times\Omega,\label{EqIstronge}\\
 &\partial_{\nu}I_{\varepsilon}=0\qquad\text{a.e. in }\R_0^+\times\partial\Omega,\label{bcIe}\\
 &S_{\varepsilon}(0,\cdot)=S_0,\qquad I_{\varepsilon}(0,\cdot)=I_0\qquad\text{a.e. in }\Omega.\label{inie}
\end{align}
\end{subequations}
\end{enumerate}
\end{Definition}
\begin{Theorem}[Existence of a weak-strong solution] 
 Let \eqref{AssIni} hold. Then, system \eqref{RepDege} possesses a weak-strong solution in terms of {\it Definition \ref{DefSole}}.
\end{Theorem}
\begin{proof} ({\it Sketch}) To begin with, we decouple the equations for the two components:
\begin{subequations}\label{Ie}
 \begin{empheq}[left={ \empheqlbrace\ }]{align}
  &\partial_t I_{\varepsilon}=(\varepsilon+\bar{S}_{\varepsilon})\Delta I_{\varepsilon}+g(\bar{S}_{\varepsilon},I_{\varepsilon})&&\text{ in }\R^+\times\Omega,\\
  &\partial_{\nu} I_{\varepsilon}=0&&\text{ in }\R^+\times\partial\Omega,\\
  &I_{\varepsilon}(0,\cdot)=I_0&&\text{ in }\Omega
 \end{empheq}
\end{subequations}
 and 
 \begin{subequations}\label{Se}
 \begin{empheq}[left={ \empheqlbrace\ }]{align}
  &\partial_t S_{\varepsilon}=\nabla\cdot\left(\nabla S_{\varepsilon}+\chi(S_{\varepsilon})S_{\varepsilon}\nabla I_{\varepsilon}\right)+f(S_{\varepsilon},I_{\varepsilon})&&\text{ in }\R^+\times\Omega,\\
  &\partial_{\nu} S_{\varepsilon}=0&&\text{ in }\R^+\times\partial\Omega,\\
  &S_{\varepsilon}(0,\cdot)=S_0&&\text{ in }\Omega.
 \end{empheq}
\end{subequations}
For smooth $\bar{S}_{\varepsilon}$ and $I_0$ standard parabolic theory \cite{LSU} insures the existence of a unique classical solution $I_{\varepsilon}$ to \eqref{Ie}. Similarly, such $I_{\varepsilon}$ and smooth $S_0$ lead to a unique classical solution $S_{\varepsilon}$ to \eqref{Se}. Moreover, it is clear from the proofs that results of {\it Lemmas \ref{Lembnd}-\ref{LemConv}}  continue to hold. They allow to obtain solutions with the regularity as stated in {\it Definition \ref{DefSole}} under assumption \eqref{AssIni} by means of an approximation procedure. 

\noindent
Uniqueness can be established in both cases in the usual way by considering  equations for  differences of two solutions, testing with suitable test functions, performing estimates, and, finally, applying the  Gronwall lemma. Here we only check uniqueness for \eqref{Se}: Let $S_{\varepsilon}^{(1)}$ and $S_{\varepsilon}^{(2)}$ be two solutions corresponding to some $I_{\varepsilon}$ and $S_0$. Set
\begin{align}
 U:=S_{\varepsilon}^{(1)}-S_{\varepsilon}^{(2)},\quad \xi_1:=\frac{\chi(S_{\varepsilon}^{(1)})S_{\varepsilon}^{(1)}-\chi(S_{\varepsilon}^{(2)})S_{\varepsilon}^{(2)}}{S_{\varepsilon}^{(1)}-S_{\varepsilon}^{(2)}}\nabla I_{\varepsilon},\quad \xi_2:=\frac{f(S_{\varepsilon}^{(1)},I_{\varepsilon})-f(S_{\varepsilon}^{(2)},I_{\varepsilon})}{S_{\varepsilon}^{(1)}-S_{\varepsilon}^{(2)}}.\nonumber
\end{align}
In this notation we have for $U$ the linear equation
\begin{align}
 \partial_t U=\Delta U+\nabla\cdot(  \xi_1 U)+\xi_2U\qquad\text{in }L^2(0,T;(H^1(\Omega))').\label{EqU}
\end{align}
We need to verify that $U=0$ a.e. Observe that
\begin{align}
 &\xi_1=K\left(1+S_{\varepsilon}^{(1)}+S_{\varepsilon}^{(2)}\right)\nabla I_{\varepsilon}\in L^{\infty}(0,T;L^2(\Omega)),\label{xi1}\\
 &\nabla{\cdot}\xi_1=K\left(1+S_{\varepsilon}^{(1)}+S_{\varepsilon}^{(2)}\right)\Delta I_{\varepsilon}+K\left(1+\nabla S_{\varepsilon}^{(1)}+\nabla S_{\varepsilon}^{(2)}\right)\cdot\nabla I_{\varepsilon}\in L^2(0,T;L^1(\Omega)),\\
 &|\xi_2|\leq \|\partial_Sf\|_{L^{\infty}((0,1)\times \R^+))}<\infty.\label{estxi2}
\end{align}

\noindent
We are going to test \eqref{EqU} with \begin{align}
\sign_{\delta}(U):=\begin{cases}\sign (U)&\text{for }|U|\geq\delta,\\
\frac{1}{\delta}U&\text{for }|U|<\delta
\end{cases}\qquad\text{for }\delta>0,\nonumber
\end{align}
and then pass to the limit as $\delta\rightarrow0$. 
Since $\sign_{\delta}\in W^{1,\infty}(\R)$, it holds that
\begin{align}
\sign_{\delta}(U)\in L^2(0,T;H^1(\Omega)),\nonumber
\end{align}
i.e., it is a valid test function. Using the weak chain and product  rules and \eqref{xi1}-\eqref{estxi2} where necessary, we thus compute that
\begin{align}
 \left<\partial_tU,\sign_{\delta}(U)\right>=\frac{d}{dt}\int_{\Omega}\int_0^U\sign_{\delta}(W)\,dWdx\underset{\delta\rightarrow0}{\rightarrow}\frac{d}{dt}\|U\|_{L^1(\Omega)}\qquad \text{in }D'(0,T),\label{conv5}
\end{align}
\begin{align}
 \int_{\Omega}\nabla U\cdot\nabla\sign_{\delta}(U)\,dx=\int_{\Omega}\sign_{\delta}'(U)|\nabla U|^2\,dx\geq0,
\end{align}
\begin{align}
 \int_{\Omega}\sign_{\delta}(U)\nabla\cdot(\xi_1 U)\,dx=&\int_{\Omega}\left(U\sign_{\delta}(U)\nabla\cdot\xi_1+\xi_1\cdot\nabla \int_0^U\sign_{\delta}(W)\,dW\right) \,dx\nonumber\\
 \underset{\delta\rightarrow0}{\rightarrow}&\int_{\Omega}\left(|U|\nabla\cdot\xi_1+\xi_1\cdot\nabla|U|\right) \,dx\qquad \text{in }D'(0,T)\nonumber\\
 =&\int_{\Omega}\nabla\cdot(\xi_1|U|)\,dx=0,
\end{align}
\begin{align}
 \int_{\Omega}\xi_2U\sign_{\delta}(U)\,dx\leq \Cl{C10}\|U\|_{L^1(\Omega)}.\label{est5}
\end{align}
Combining \eqref{conv5}-\eqref{est5} we obtain that
\begin{align}
 \frac{d}{dt}\|U\|_{L^1(\Omega)}\leq \Cr{C10}\|U\|_{L^1(\Omega)}.\label{est10}
\end{align}
Finally, applying the Gronwall lemma to \eqref{est10}, we conclude that $\|U\|_{L^1(\Omega)}\equiv0$.

\noindent
Altogether, we have a well-defined operator $\Phi$ in the following setting:
\begin{align}
&X:=L^2(0,T;L^2(\Omega)),\qquad {\cal M}:=\{S\in X:\ 0\leq S\leq 1\},\nonumber\\
 &\Phi: {\cal M}\rightarrow {\cal M},\qquad \Phi(\bar{S}_{\varepsilon}):=S_{\varepsilon}\nonumber
\end{align}
Thanks to  {\it Lemma \ref{LemConv}} the  image  $\Phi({\cal M})$ is precompact in $X$. Moreover, this Lemma together with fact that both equations are uniquely solvable imply the continuity of $\Phi$.
Therefore, the Schauder fixed point theorem implies the existence of a fixed point and, as a result, of a solution to \eqref{RepDege}. 


\end{proof}
\section{Existence of  solutions to \texorpdfstring{\eqref{RepDeg}}{} { for} \texorpdfstring{$N=1$ and $S_0>0$}{}}\label{Ex1D}
{ Throughout} this section we assume that
\begin{align}
 \Omega\text{ is a finite interval in }\R\label{OmegaInt}
\end{align}
and 
\begin{align}
 0<\inf_{\Omega}S_0\leq S_0\leq 1.\label{PosS0}
\end{align}
\subsection{An a priori lower bound}
To begin with, { we return to} the regularised problem \eqref{RepDege} and establish an a priori uniform positive lower bound for $S_{\varepsilon}$.
\begin{Lemma}
 Solutions to \eqref{RepDege} satisfy
 \begin{align}
  \left\|S_{\varepsilon}^{-1}\right\|_{L^{\infty}(0,T;L^{\infty}(\Omega))}\leq\C(T).\label{estSm1}
 \end{align}

\end{Lemma}
\begin{proof}
 We use the standard method of propagation of $L^p$ bounds in order to derive a finite uniform upper bound for $S_{\varepsilon}^{-1}$. Its inverse gives a uniform positive lower bound for $S_{\varepsilon}$. Let $p\geq1$. Multiplying \eqref{EqSe} by $-pS_{\varepsilon}^{-p-1}$ and integrating by parts over $\Omega$ we obtain using the H\"older, Young, and { Gagliardo-Nirenberg} inequalities as well as estimate \eqref{estnabIe} where necessary that
 \begin{align}
  \frac{d}{dt}\left\|S_{\varepsilon}^{-\frac{p}{2}}\right\|_{L^2(\Omega)}^2
  =&-\frac{4(p+1)}{p}\left\|\nabla S_{\varepsilon}^{-\frac{p}{2}}\right\|_{L^2(\Omega)}^2+2(p+1)\int_{\Omega}\chi(S_{\varepsilon})S_{\varepsilon}^{-\frac{p}{2}}\nabla S_{\varepsilon}^{-\frac{p}{2}}\cdot \nabla I_{\varepsilon}\,dx\nonumber\\
  &-p\int_{\Omega}f(S_{\varepsilon},I_{\varepsilon})S_{\varepsilon}^{-p-1}\,dx\nonumber\\
  \leq& -\C\left\|\nabla S_{\varepsilon}^{-\frac{p}{2}}\right\|_{L^2(\Omega)}^2+\C p\left\|\nabla S_{\varepsilon}^{-\frac{p}{2}}\right\|_{L^2(\Omega)}\left\| S_{\varepsilon}^{-\frac{p}{2}}\right\|_{L^{\infty}(\Omega)}\|\nabla I_{\varepsilon}\|_{L^2(\Omega)}\nonumber\\
  &+p\lambda_S \left\|S_{\varepsilon}^{-1}\right\|_{L^p(\Omega)}^p\nonumber\\
  \leq& -\Cl{C1}\left\|\nabla S_{\varepsilon}^{-\frac{p}{2}}\right\|_{L^2(\Omega)}^2+\C p^2\left\| S_{\varepsilon}^{-\frac{p}{2}}\right\|_{L^{\infty}(\Omega)}^2
  \nonumber\\
  \leq& -\Cr{C1}\left\|S_{\varepsilon}^{-\frac{p}{2}}\right\|_{H^1(\Omega)}^2+\C p^2\left\| S_{\varepsilon}^{-\frac{p}{2}}\right\|_{L^{\infty}(\Omega)}^2\nonumber\\
  \leq& -\Cr{C1}\left\|S_{\varepsilon}^{-\frac{p}{2}}\right\|_{H^1(\Omega)}^2+\C p^2\left\|S_{\varepsilon}^{-\frac{p}{2}}\right\|_{H^1(\Omega)}^{\frac{4}{3}}\left\| S_{\varepsilon}^{-\frac{p}{2}}\right\|_{L^1(\Omega)}^{\frac{2}{3}}\nonumber\\
  \leq& \Cl{C5} p^6\left\| S_{\varepsilon}^{-\frac{p}{2}}\right\|_{L^1(\Omega)}^2\label{estL1}
  \\
  \leq& \Cl{C3} p^6\left\| S_{\varepsilon}^{-\frac{p}{2}}\right\|_{L^2(\Omega)}^2.\label{estL2}
 \end{align}
Using the Gronwall lemma we conclude from \eqref{estL2} that
\begin{align}
 \left\|S_{\varepsilon}^{-1}(t,\cdot)\right\|_{L^p(\Omega)}^p\leq &e^{t\Cr{C3} p^6}\left\|S_{0}^{-1}\right\|_{L^p(\Omega)}^p
 \leq e^{t\Cr{C3} p^6}|\Omega|\left\|S_{0}^{-1}\right\|_{L^{\infty}(\Omega)}^p.\nonumber
\end{align}
Consequently,
\begin{align}
 \left\|S_{\varepsilon}^{-1}\right\|_{L^{\infty}(0,T;L^p(\Omega))}\leq \C(T,p)\qquad\text{for all }p\geq1.\label{estLp}
\end{align}
Further, integrating \eqref{estL1} over $(0,t)$ we obtain that
\begin{align}
 \left\|S_{\varepsilon}^{-1}(t,\cdot)\right\|_{L^p(\Omega)}^p\leq &\left\|S_{0}^{-1}\right\|_{L^p(\Omega)}^p+\Cr{C5} p^6\int_0^t\left\|S_{\varepsilon}^{-1}({ s},\cdot)\right\|_{L^{\frac{p}{2}}(\Omega)}^p\,d s.\nonumber
\end{align}
Consequently, 
\begin{align}
 \left\|S_{\varepsilon}^{-1}\right\|_{L^{\infty}(0,T;L^p(\Omega))}^p\leq |\Omega|\left\|S_{0}^{-1}\right\|_{L^{\infty}(\Omega)}^p+T\Cr{C5} p^6\left\|S_{\varepsilon}^{-1}\right\|_{L^{\infty}(0,T;L^{\frac{p}{2}}(\Omega))}^p.\label{estrec1}
\end{align}
For $n\in\N$ we introduce 
\begin{align}
 A_n:=\left\|S_{0}^{-1}\right\|_{L^{\infty}(\Omega)}^{2^{n+1}}+\left\|S_{\varepsilon}^{-1}\right\|_{L^{\infty}(0,T;L^{2^{n+1}}(\Omega))}^{2^{n+1}}.\nonumber
\end{align}
Due to estimate \eqref{estrec1} we have that
\begin{align}
 A_{n+1}\leq\Cl{C4}(T)(2^{6})^n A_{n}^2,\label{estArec}
\end{align}
whereas \eqref{estLp} implies that
\begin{align}
 A_0\leq \Cl{CA0}(T).\label{estA0}
\end{align}
Using the standard recursive result 
\cite[Chapter 2, Lemma 5.6]{LSU} 
we conclude with \eqref{estArec} that for all $n\in\N$
\begin{align}
 A_n\leq &\Cr{C4}^{2^{n}-1}(T)(2^{6})^{2^{n}-1-n}A_0^{2^{n}}.\label{estAn}
\end{align}
Combining \eqref{estA0}-\eqref{estAn} we conclude that
\begin{align}
 \left\|S_{\varepsilon}^{-1}\right\|_{L^{\infty}(0,T;L^{2^{n+1}}(\Omega))}\leq A_n^{2^{-(n+1)}}\leq \C(T).\label{estuni}
\end{align}
Since $\|\cdot\|_{2^{n+1}}\rightarrow\|\cdot\|_{\infty}$ as $n\rightarrow\infty$, \eqref{estuni} implies \eqref{estSm1} .
\end{proof}

\subsection{Existence of solutions}
Having { obtained the} estimate \eqref{estSm1} we can now prove the existence of weak-strong solutions to \eqref{RepDeg}.  The definition is as follows:
\begin{Definition}[Weak-strong solution]\label{DefSol} We call a pair of measurable functions $(S,I):\R_0^+\times \overline{\Omega}\rightarrow[0,1]\times \R_0^+$ a \underline{weak-strong solution} to system \eqref{RepDeg} if: 
\begin{enumerate} 
  \item $\nabla S\in L^2_{loc}(\R_0^+;L^2(\Omega))$, $\partial_t S\in L^2_{loc}(\R_0^+;(H^1(\Omega))')$;
  \item $I\in L^2_{loc}(\R_0^+;H^2(\Omega))$, $\nabla I\in L^{\infty}(\R_0^+;L^2(\Omega))$, $\partial_t I\in L^2_{loc}(\R_0^+;L^2(\Omega))$;
 \item  the pair $(S,I)$ is a weak solution to  \eqref{EqS} and a  strong solution to \eqref{EqI}, i.e., 
 for all $\varphi\in H^1(\Omega)$  it holds that
 \begin{subequations}
 \begin{align}
  &\left<\partial_t S,\varphi\right>=-\int_{\Omega}(\nabla S+\chi(S)S\nabla I)\cdot\nabla\varphi \,dx+{ \int _{\Omega}}f(S,I)\varphi\,dx\qquad\text{a.e. in }\R_0^+,\label{EqSweakN1}\\
 &\partial_t I =S\Delta I +g(S,I)\qquad\text{a.e. in }\R_0^+\times\Omega,\label{EqIstrong}\\
 &\partial_{\nu}I=0\qquad\text{a.e. in }\R_0^+\times\partial\Omega,\label{bcI}
 \\
 &S(0,\cdot)=S_0,\qquad I(0,\cdot)=I_0\qquad\text{a.e. in }\Omega.
\end{align}
\end{subequations}
\end{enumerate}

\end{Definition}
\begin{Theorem}[Existence of a weak-strong solution]
 Under assumptions \eqref{AssIni} and  \eqref{OmegaInt}-\eqref{PosS0} system \eqref{RepDeg} possesses a weak-strong solution in terms of {\it Definition \ref{DefSol}}.
\end{Theorem}
\begin{proof}
 Our starting point is the weak formulation from {\it Definition \ref{DefSole}} and {\it Lemma \ref{LemConv}} on convergence.  Thanks to the uniform estimates \eqref{estDeltaIe} and \eqref{estSm1} and the Banach-Alaoglu theorem we may assume that the sequence from that Lemma is chosen in such a way that  
\begin{alignat}{3}
 &\Delta I_{\varepsilon_n}\underset{n\rightarrow\infty}{\rightharpoonup}\Delta I&&\qquad \text{in }L^2(0,T;L^2(\Omega))\label{convDeltaIe}
\end{alignat}
holds as well. 
Using \eqref{estnabIe}-\eqref{estSet} and \eqref{convDeltaIe}   together with the dominated convergence theorem and compensated compactness, we can pass to the limit in \eqref{DefWStre} along the sequence  $\varepsilon_n\rightarrow 0$ and thus  obtain that $(S,I)$ satisfies all conditions from {\it Definition \ref{DefSol}}.
\end{proof}

\section{Existence of  supersolutions to \texorpdfstring{\eqref{RepDeg}}{} { for} \texorpdfstring{$N\geq2$ and $S_0\geq0$}{}}\label{Gensuper}
\noindent
In this Section we consider the general case of an arbitrary space dimension, { assume} $0\leq S_0\leq 1$, and prove the existence of a weak supersolution to \eqref{RepDeg}. The definition is as follows:
\begin{Definition}[Weak supersolution]\label{DefGenSuperSol}We call a pair of measurable functions $(S,I):\R_0^+\times \overline{\Omega}\rightarrow[0,1]\times \R_0^+$ a \underline{weak supersolution} to system \eqref{RepDeg} if: 
\begin{enumerate} 
  \item $\nabla S\in L^2_{loc}(\R_0^+;L^2(\Omega))$, $\partial_t S\in L^2_{loc}(\R_0^+;(H^1(\Omega))')$;
  \item $I\in L^2_{loc}(\R_0^+;L^2(\Omega))$, $\nabla I\in L^{\infty}(\R_0^+;L^2(\Omega))$, $\partial_t I\in L^2_{loc}(\R_0^+;L^2(\Omega))$;
 \item  the pair $(S,I)$ is a weak solution to  \eqref{EqS} and a weak supersolution to \eqref{EqI}, i.e., 
 for all $\varphi\in H^1(\Omega)$ and $0\leq\psi\in W^{1,\infty}(\Omega)$  it holds that
 \begin{subequations}\label{EqSuperweak}
 \begin{align}
  &\left<\partial_t S,\varphi\right>=-\int_{\Omega}(\nabla S+\chi(S)S\nabla I)\cdot\nabla\varphi\,dx+{ \int_{\Omega}} f(S,I)\varphi\,dx\qquad\text{a.e. in }\R_0^+,\label{EqSweak}\\
 &\int_{\Omega}\partial_t I \psi\,dx
 \geq\int_{\Omega}\left(-\nabla I\cdot \nabla(\psi S)+g(S,I)\psi\right) \,dx\qquad\text{a.e. in }\R_0^+,\label{EqIweak}\\
 &S(0,\cdot)=S_0,\qquad I(0,\cdot)=I_0\qquad\text{a.e. in }\Omega.\label{inisuper}
\end{align}
\end{subequations}
\end{enumerate}
\end{Definition}
\begin{Remark}
 Recently weak (generalised) supersolutions in the form of a variational inequality have been used in order to provide a solution concept for models with positive chemotaxis, see e.g. \cite{LankWink2017,ZhigunKSLog,ZhigunKS}. In those cases, however,  both equations are in divergent form, while equation \eqref{EqI} is not. The latter precludes the possibility to close  \eqref{EqSuperweak} by imposing a   suitable mass control from above. As a result, even a smooth supersolution $(S,I)$ is not automatically a subsolution to \eqref{RepDeg}.    
\begin{Theorem}[Existence of a weak supersolution] Let \eqref{AssIni} hold. Then, system \eqref{RepDeg} possesses a weak supersolution in terms of {\it Definition \ref{DefGenSuperSol}}.
\end{Theorem}
\begin{proof}
 Once again, our starting point is the weak formulation from {\it Definition \ref{DefSole}} and {\it Lemma \ref{LemConv}} on convergence. 
To begin with, we construct a suitable reformulation of \eqref{EqIe}. Since this equation is not in divergence form, the standard approach based on testing and integration by parts is not a good foundation for a limit procedure. It turns out useful to construct instead a variational identity which combines equations for both solution components.  

\noindent
Let $0\leq\eta\in L^{\infty}(0,T;W^{1,\infty}(\Omega))$, so that $\eta I_{\varepsilon}\in  L^2(0,T;H^1(\Omega))$. Testing \eqref{EqSweake} and { \eqref{EqIe}} with $-\eta I_{\varepsilon}$ and  $\eta$, respectively, adding  the results, integrating over $(0,T)$, and using the chain rule where necessary, we compute 
\begin{align}
 &\int_0^T\left(-\left<\partial_tS_{\varepsilon},\eta I_{\varepsilon}\right>+\int_{\Omega}\eta\partial_tI_{\varepsilon}\,dx\right) dt\nonumber\\
 =&\int_0^T\int_{\Omega}\left(\left(\nabla S_{\varepsilon}+\chi(S_{\varepsilon})S_{\varepsilon}\nabla I_{\varepsilon}\right)\cdot\nabla (\eta I_{\varepsilon})-\nabla I_{\varepsilon}\cdot \nabla\left(\eta (\varepsilon+S_{\varepsilon})\right)\right) \,dxdt\nonumber\\
 &+\int_0^T\int_{\Omega}\left(-f(S_{\varepsilon},I_{\varepsilon})I_{\varepsilon}+g(S_{\varepsilon},I_{\varepsilon})\right)\eta\,dxdt\nonumber\\
 =&\int_0^T\int_{\Omega}\left(\eta\chi(S_{\varepsilon})S_{\varepsilon}|\nabla I_{\varepsilon}|^2+\left(I_{\varepsilon}\nabla S_{\varepsilon} -(\varepsilon+S_{\varepsilon})\nabla I_{\varepsilon}+\chi(S_{\varepsilon})S_{\varepsilon}I_{\varepsilon}\nabla I_{\varepsilon}\right)\cdot \nabla\eta\right) \,dxdt\nonumber\\
 &+\int_0^T\int_{\Omega}\left(-f(S_{\varepsilon},I_{\varepsilon})I_{\varepsilon}+g(S_{\varepsilon},I_{\varepsilon})\right)\eta\,dxdt.\label{Vareps}
\end{align}
Using \eqref{convnabIe}-\eqref{convSet} together with the dominated convergence theorem and the compensated compactness, we can pass to the limit in \eqref{EqSweake} and \eqref{inie} along the sequence  $\varepsilon_n\rightarrow 0$ which yields \eqref{EqSweak} and \eqref{inisuper}, respectively. 
Owing to the presence of  the quadratic term $|\nabla I_{\varepsilon}|^2$ in one of the integrals, we cannot justify the equality while passing to the limit in \eqref{Vareps}. Instead, we take limit inferior on both sides, use the above mentioned convergences and theorems, as well as the weak lower semicontinuity of a norm, and, finally, the chain rule where necessary and thus arrive at
\begin{align}
 & \int_0^T\left(-\left<\partial_t S,\eta I\right>+\int_{\Omega}\eta{ \partial_tI}\,dx \right) dt\nonumber\\
 \geq& \int_0^T\int_{\Omega}\left(\eta\chi(S)S|\nabla I|^2+\left(I\nabla S -S\nabla I+\chi(S)SI\nabla I\right)\cdot \nabla\eta+\left(-f(S,I)I+g(S,I)\right)\eta\right) \,dxdt\nonumber\\
 =&\int_0^T\int_{\Omega}\left(\left(\nabla S+\chi(S)S\nabla I\right)\cdot\nabla(\eta I)-\nabla I\cdot \nabla(\eta S)+\left(-f(S,I)I+g(S,I)\right)\eta\right) \,dxdt.\label{Varcomb}
\end{align}
Plugging ${ \varphi=}\eta I\in L^2(0,T;H^1(\Omega))$ into \eqref{EqSweake}, integrating over $(0,T)$, { passing to the limit for $\varepsilon_n\to 0$,} and then adding the result to  \eqref{Varcomb} finally yields \eqref{EqIweak}.
\end{proof}

\section{Numerical simulations and discussion}\label{sec:nums}

\noindent
In order to illustrate the solution behavior we present in this section some 2D numerical simulation results for the system \eqref{RepDeg}.
A finite difference scheme with a first order upwind discretization of the repellent taxis term was used to produce them. Here we show contour plots for the populations of susceptibles/infected at various time points in a square domain $\Omega =[0,10]^2$. We consider complementary initial densities $$I_0 = \sum_{i,j=1}^{3}C_i \exp(-(x^2_j + y^2_j)/2\epsilon),\qquad S_0 = 1-I_0 $$
and use parameters given in {\it Table \ref{TabPar}}.
\noindent
\begin{table}[h]
\centering
\begin{tabular}{c|c|c|c|c|c|c|c}
 $C_1$&$C_2$&$C_3$&$\epsilon$&$\lambda_S$&$\lambda_I$& $\mu_I$&$\mu_S$\\\hline
 $0.1$&$0.2$&$0.3$&$0.25$&$0.5$&$0.5$&$0.05$&$0.01$
\end{tabular}\vspace{-1em}
\caption{\footnotesize Parameters}\vspace{-1em}
\label{TabPar}
\end{table}
{\it Figure \ref{both}}  shows simulations of the model in the cases with ($\chi=15$) and without ($\chi=0$) repellent taxis, respectively. The former exhibits a spread of the infection comparable with the latter case, but with a reduced suppression of the susceptibles and an overall slightly higher infected population density. Due to the infectives' diffusion with contact, however, the avoidance efficiency is diminished, which for even larger values of $\chi $ leads to a more effective spread of the invasion. Further numerical simulations (not shown in this paper) suggest the existence of a critical value of $\chi$ above which the repellent taxis actually triggers the opposite effect. Moreover, here the sensitivity of susceptibles towards infected was taken to be linearly decreasing  with $S$ 
which also contributed to the mentioned infection enhancement, but it is reasonable to assume its dependence also on $I$, more precisely on the interactions between the two populations. Analytically determining the critical $\chi $ range and its effect on the epidemic spread would be an interesting problem in the framework of travelling wave analysis. 
\printbibliography
\begin{figure}[h]
\centering
\begin{tabular}{ccccc}
 {\rotatebox{90}{\footnotesize\phantom{ttt}infected}}&\includegraphics[scale=0.23]{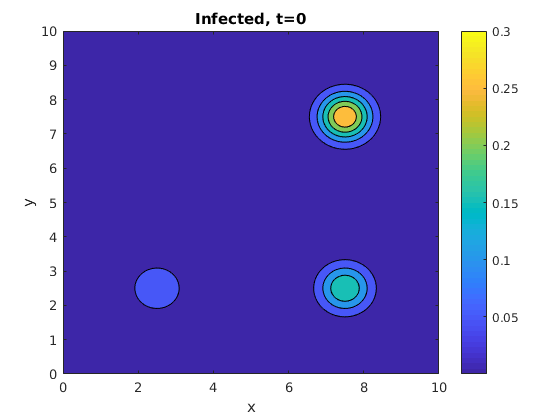}\
&\includegraphics[scale=0.23]{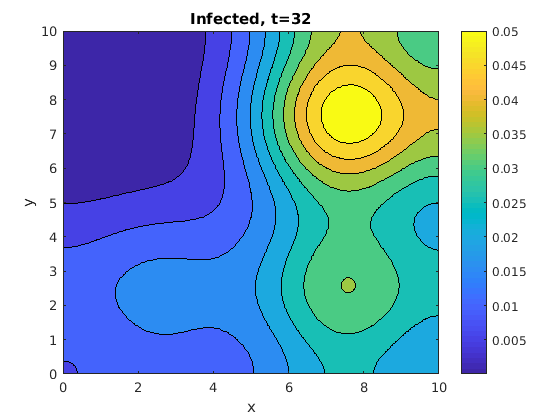}\ &\includegraphics[scale=0.23]{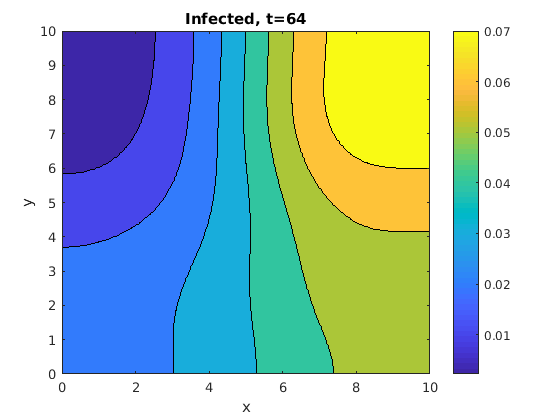}\ &\includegraphics[scale=0.23]{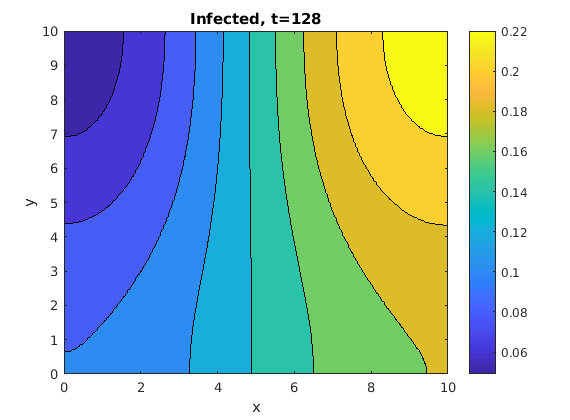}\\
{\rotatebox{90}{\footnotesize\phantom{ttt}susceptibles}}&\includegraphics[scale=0.23]{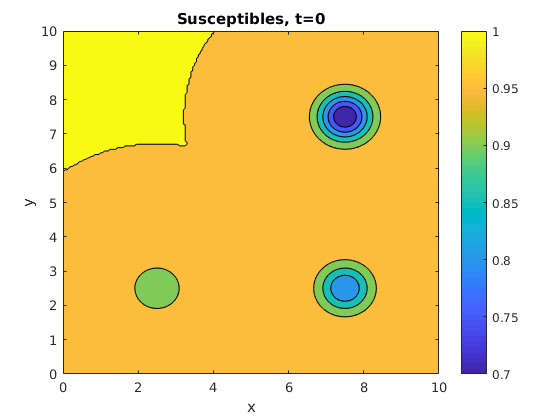}\
&\includegraphics[scale=0.23]{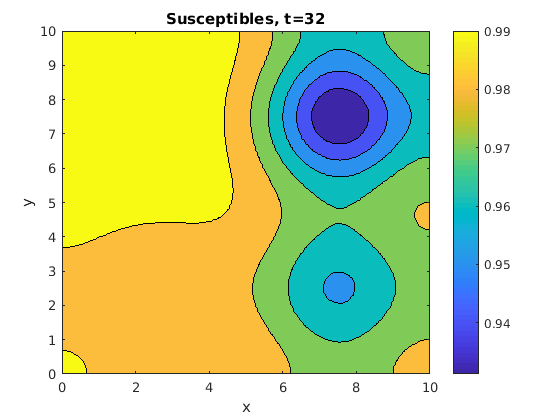}\ &\includegraphics[scale=0.23]{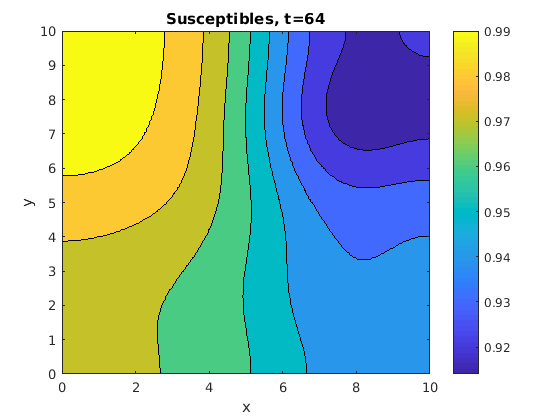}\ &\includegraphics[scale=0.23]{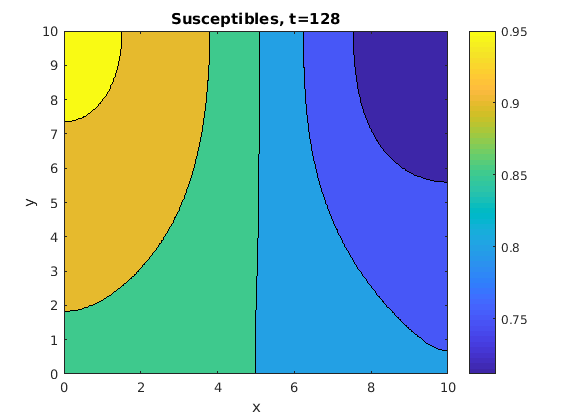}
\end{tabular}
%
%
\begin{tabular}{ccccc}
  {\rotatebox{90}{\footnotesize\phantom{ttt}infected}}&\includegraphics[scale=0.23]{initialsI}\	&\includegraphics[scale=0.23]{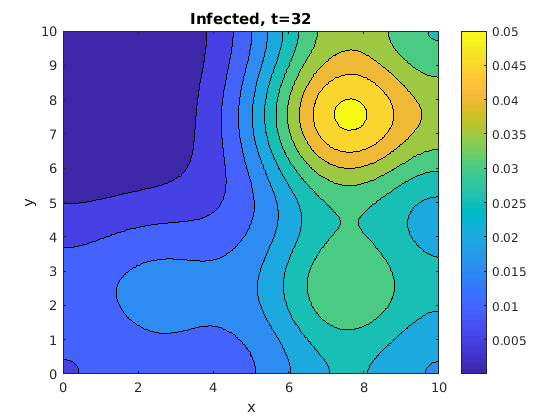}\ &\includegraphics[scale=0.23]{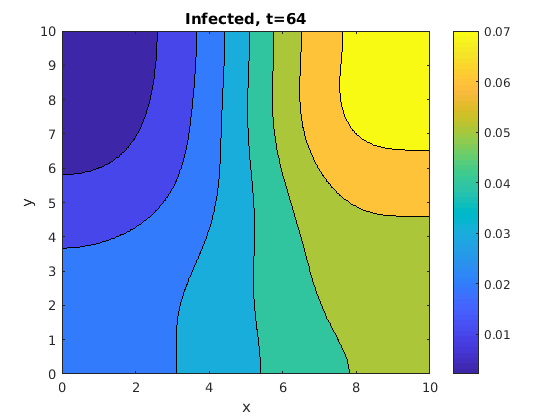}\ &\includegraphics[scale=0.23]{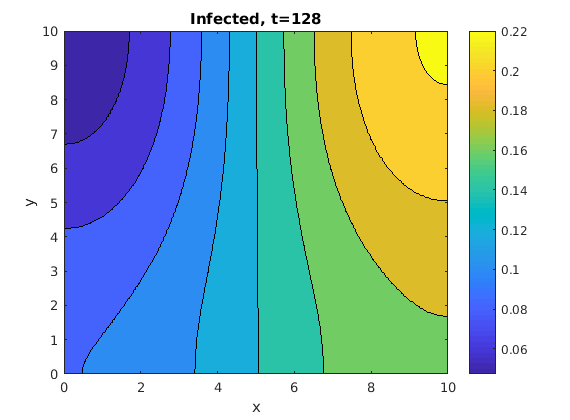}\\
{\rotatebox{90}{\footnotesize\phantom{ttt}susceptibles}}	&\includegraphics[scale=0.23]{initialsS}\	&\includegraphics[scale=0.23]{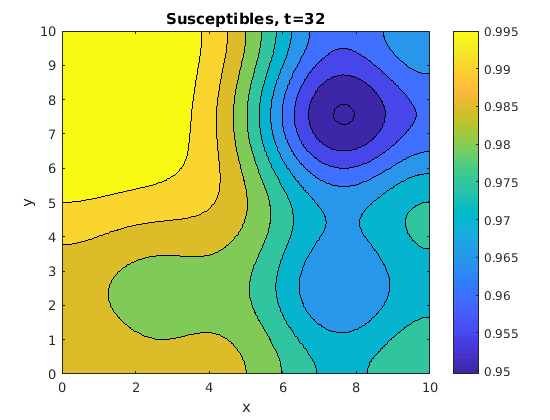}\ &\includegraphics[scale=0.23]{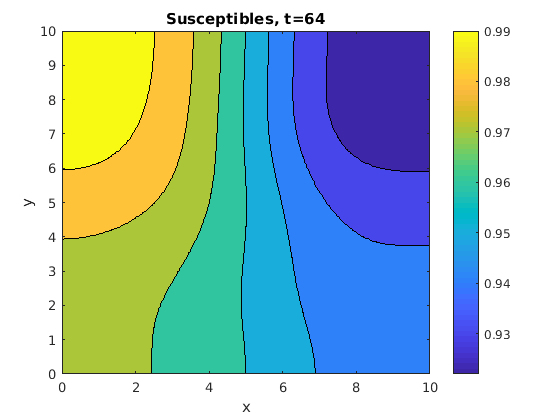}\ &\includegraphics[scale=0.23]{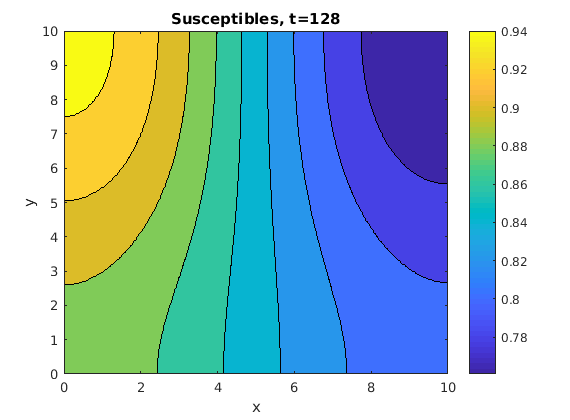}
\end{tabular}\vspace{-1em}
	\caption{\footnotesize Simulation results for $\chi=15$ (upper two rows) and $\chi =0$ (lower rows)
	at different times.} \vspace{-1.5em}
\label{both}
\end{figure}


\end{document}